\documentclass[oneside,english]{amsart}

\usepackage[T1]{fontenc}
\usepackage[latin9]{inputenc}
\usepackage{babel}
\usepackage{verbatim}
\usepackage{prettyref}
\usepackage{mathtools}
\usepackage{amstext}
\usepackage{amsthm}
\usepackage{amssymb}
\usepackage{xcolor}
\usepackage{ulem}

\usepackage[unicode=true,pdfusetitle,
 bookmarks=true,bookmarksnumbered=false,bookmarksopen=false,
 breaklinks=false,pdfborder={0 0 1},backref=false,colorlinks=false]
 {hyperref}

\makeatletter
\numberwithin{equation}{section}
\numberwithin{figure}{section}
\theoremstyle{plain}
\newtheorem{thm}{\protect\theoremname}[section]
\theoremstyle{definition}
\newtheorem{defn}[thm]{\protect\definitionname}
\theoremstyle{plain}
\newtheorem{prop}[thm]{\protect\propositionname}
\theoremstyle{remark}
\newtheorem{rem}[thm]{\protect\remarkname}
\theoremstyle{plain}
\newtheorem{lem}[thm]{\protect\lemmaname}
\theoremstyle{definition}

\theoremstyle{plain}

\usepackage[T1]{fontenc}    
\def\monthenglish{\ifcase\month \or
  January\or February\or March\or April\or May\or June\or
  July\or August\or September\or October\or November\or December\fi}
\usepackage{a4wide}        
\usepackage{amsmath}
\usepackage{euscript}
\usepackage{enumerate}
\usepackage{prettyref}
\allowdisplaybreaks[1] 
\usepackage{times,mathptm} 
\usepackage{amsfonts}
\newcommand{\hide}[1]{}

\newcommand{\N}{\mathbb{N}}
\newcommand{\R}{\mathbb{R}}

\newcommand*{\dom}{\operatorname{dom}}            

\newcommand*{\e}{\mathrm{e}}

\DeclareMathAccent{\Circ}{\mathalpha}{operators}{"17}
\newcommand{\interior}[1]{\Circ{#1}}

\renewcommand{\Re}{\operatorname{\mathfrak{Re}}}

\newcommand{\rk}{\operatorname{rank}}

\renewcommand{\tilde}{\widetilde}
\renewcommand*{\epsilon}{\varepsilon}
\renewcommand{\d}{\,\mathrm{d}}

\newrefformat{prop}{Proposition \ref{#1}}
\newrefformat{lem}{Lemma \ref{#1}}
\newrefformat{thm}{Theorem \ref{#1}}
\newrefformat{cor}{Corollary \ref{#1}}
\newrefformat{rem}{Remark \ref{#1}}
\newrefformat{exa}{Example \ref{#1}}
\newrefformat{sub}{Subsection \ref{#1}}
\newrefformat{eq}{(\ref{#1})}

\makeatother

\providecommand{\corollaryname}{Corollary}
\providecommand{\definitionname}{Definition}
\providecommand{\examplename}{Example}
\providecommand{\lemmaname}{Lemma}
\providecommand{\propositionname}{Proposition}
\providecommand{\remarkname}{Remark}
\providecommand{\theoremname}{Theorem}

\begin{document}
\title{A Structural Observation on Port-Hamiltonian Systems}
\author{Rainer Picard, Sascha Trostorff, Bruce Watson \& Marcus Waurick}
\maketitle

\textbf{Abstract.} We study port-Hamiltonian systems on a familiy of intervals and characterise all boundary conditions leading to $m$-accretive realisations of the port-Hamiltonian operator and thus to generators of contractive semigroups. The proofs are based on a structural observation that the port-Hamiltonian operator can be transformed to the derivative on a familiy of reference intervals by suitable congruence relations allowing for studying the simpler case of a transport equation. Moreover, we provide well-posedness results for associated control problems without assuming any additional regularity of the operators involved.   

\section{Introduction}

In this paper, we shall revisit Port-Hamiltonian differential equations (going back to van der Schaft et al. \cite{van_der_Schaft2006, van_der_Schaft2002}), 
that is, a system of first order partial differential equations of
the form
\[
\begin{cases}
\partial_{t}u+P_{1}\partial_{x}\mathcal{H}u+P_{0}\mathcal{H}u=0 & \text{on }]0,\infty[\times I\\
u(0,x)=u_{0}(x), & x\in I,
\end{cases}
\]
where $I$ is a 
 real interval, $u\colon]0,\infty[\times I\to\mathbb{R}^{n}$
is a 
vector field 
subject to suitable (linear)
boundary conditions, $\mathcal{H}\colon I\to\mathbb{R}^{n\times n}$
is a matrix field attaining values in the symmetric positive definite
matrices, $P_{1}=P_{1}^{*}\in\mathbb{R}^{n\times n}$ is invertible,
$P_{0}=-P_{0}^{*}\in\mathbb{R}^{n\times n}$. \footnote{The theory developed in this article also works for complex matrices and complex-valued functions. However, since the complex case can always be reduced to the real case by considering copies of real spaces, we restrict ourselves to the real case.}
There is a vast amount of literature
addressing the well-posedness as well as other questions related to
the equations at hand (see e.g. the monograph \cite{Jacob_Zwart_ISEM}, the survey \cite{Jacob_Zwart_2018} as well as the PhD thesis \cite{Augner2016} and the references therein). In particular, questions in the theory of boundary
control and observations are treated in the framework of port-Hamiltonian
systems. Also, higher-dimensional variants of port-Hamiltonian systems 
or port-Hamiltonian systems of higher order are discussed (see e.g. \cite{LeGorrec2006,Augner2016,Jacob2019}). As an intermediate step, several authors have dealt
with port-Hamiltonian systems on networks (see \cite{Jacob_Morris_Zwart2015,Jacob2019,Jacob_Wegner2019}, where networks are considered as examples and \cite{Waurick_Wegner_2020} for a detailed study). More precisely, the interval
$I$ is replaced by a set of intervals. In this case, the role of
boundary conditions becomes more pronounced. In particular, as a result,
there is an abundance of descriptions for boundary conditions leading
to port-Hamiltonian operators, that is, 
\[
-P_{1}\partial_{x}\mathcal{H}-P_{0}\mathcal{H}
\]
on a Hilbert space consisting of suitably many copies of $L^{2}$-type
spaces that generate bounded or (quasi-) contractive semigroups (see \cite{Jacob_Morris_Zwart2015} and \cite{Trostorff2014} for nonlinear boundary conditions). If
the operator $\mathcal{H}$ satisfies uniform boundedness conditions
(from above and below) then 
 $\mathcal{H}=1$ can be assumed with
no loss of generality, the desired boundary conditions result from
a subtle interplay of $P_{1}$ and $\partial_{x}$. Note that $P_{0}$
is then dealt with by a standard perturbation argument. It appears
to be commonly understood that reducing the port-Hamiltonian operator
to 
\[
-P_{1}\partial_{x}
\]
is the optimal way of treating port-Hamiltonian systems. The main
tool provided in the paper at hand is the transformation of the latter operator
 (by suitable congruence transformations) 
 to
\[
\partial_{x},
\]
arguably, the easiest case in which to discuss boundary conditions on networks.
We 
do not rely on semi-group theory as our method to show existence,
uniqueness and continuous dependence on the data, and as such we can 
address well-posedness of equations of the form 
\[
\left(\partial_{t}M_{0}+M_{1}+P_{1}\partial_{x}+P_{0}\right)u=f,
\]
on several copies of $L^{2}$-type spaces by explicit reduction 
to the case of
\[
\left(\partial_{t}\tilde{M}_{0}+\tilde{M}_{1}+\partial_{x}\right)\tilde{u}=\tilde{f}.
\]
Using the theory of evolutionary equations (see \cite{Picard2009,ISEM2020} and \cite[Chapter 6]{Picard_McGhee2011})
the latter equation, we shall furthermore be able to treat partial-differential-algebraic
equations; that is, we may allow $\tilde{M}_{0}$ to have a proper
nullspace, thus generalising the class of port-Hamiltonian systems
significantly.

 The article is structured as follows. We consider the operator $\partial_x$ on networks in \prettyref{sec:network} and provide a characterisation of all (linear) boundary condition leading to $m$-accretive realisations of this operator on a suitable Hilbert space (and hence $-\partial_x$ would generate a contraction semigroup). After that, we show in \prettyref{sec:The-Congruence-of} how the abstract port-Hamiltonian operator $P_1\partial_x\mathcal{H}$ can be reduced to the case treated in \prettyref{sec:network}, which allows us to provide a new proof for the well-posedness of port-Hamiltonian systems in \prettyref{sec:well-posedness}. Moreover, using the framework of evolutionary equations instead of $C_0$-semigroups, we present a new approach to boundary control problems, which has the benefit that one does not need to assume smooth diagonalisability of $\mathcal{H}$, which is a standard assumption in the existing literature (see e.g. \cite{Zwart2010}).

\section{The Operator $\partial_{x}$ on Networks}\label{sec:network}

In this section, we briefly introduce the main operator of this manuscript.
For this let $I_{k}\subseteq\mathbb{R}$ be a non-empty interval with
non-empty complement (i.e. $I_k\ne ]-\infty,\infty[$), where $k\in\{1,\ldots,N\}$ for some $N\in\mathbb{N}$. 
\begin{defn}
\label{def:derivative}We define
\begin{align*}
\partial_{x}\colon\bigoplus_{k=1}^{N}H^{1}(I_{k})\subseteq\bigoplus_{k=1}^{N}L^{2}(I_{k}) & \to\bigoplus_{k=1}^{N}L^{2}(I_{k})\\
(\phi_{k})_{k} & \mapsto(\phi_{k}')_{k},
\end{align*}
where $H^{1}(I_{k})$ is the (standard) Sobolev space of $L^{2}(I_{k})$-functions
with weak derivative representable as $L^{2}(I_{k})$-function. 
\end{defn}

With this operator at hand, we can consider the port-Hamiltonian operator
$P_{1}\partial_{x}$ for a suitable matrix $P_{1}\in\R^{N\times N}$
on $\bigoplus_{k=1}^{N}L^{2}(I_{k}).$ In order to get a well-defined
object, we have to restrict the class of possible matrices $P_{1}$.
\begin{defn}
A matrix $P_{1}\in\R^{N\times N}$ is called \emph{compatible, }if
$P_{1}$ leaves the space $\bigoplus_{k=1}^{N}L^{2}(I_{k})$ invariant. 
\end{defn}

\begin{rem}
 A typical example for a compatible matrix $P_1$ is a diagonal matrix. More generally, if we have several copies of one interval, say $I_1=\ldots=I_j$ for some $j\in\{1,\ldots,N\}$, then $P_1$ could be block-diagonal. However, note that the class of compatible matrices is bigger, and hence, allows for several couplings between the equations on each interval.
\end{rem}

Before we come to a closer analysis of the port-Hamiltonian operator,
we shall reduce the operator just introduced to a more managable reference
case. For this, we put
\begin{align*}
N_{\textnormal{f}} & \coloneqq\{k\in\{1,\ldots,N\}\,;\,I_{k}\text{ bounded}\},\\
M_{+} & \coloneqq\{k\in\{1,\ldots,N\}\,;\,\sup I_{k}=\infty\},\\
M_{-} & \coloneqq\{k\in\{1,\ldots,N\}\,;\,\inf I_{k}=-\infty\}.
\end{align*}
Moreover, for $n,m_{+},m_{-}\in\N$ we define the space
\[
L^{2}(n,m_{+}m_{-})\coloneqq\left(L^{2}(]-1/2,1/2[)\right)^{n}\oplus\left(L^{2}(]-1/2,\infty[)\right)^{m_{+}}\oplus\left(L^{2}(]-\infty,1/2[)\right)^{m_{-}}.
\]
Correspondingly, we introduce
\begin{align*}
H^{1}(n,m_{+},m_{-}) & \coloneqq\left(H^{1}\left(\,]-1/2,1/2[\,\right)\right)^{n}\oplus\left(H^{1}\left(\,]-1/2,\infty[\,\right)\right)^{m_{+}}\oplus\left(H^{1}\left(\,]-\infty,1/2[\,\right)\right)^{m_{-}},\\
H_0^{1}(n,m_{+},m_{-}) & \coloneqq\left(H_{0}^{1}\left(\,]-1/2,1/2[\,\right)\right)^{n}\oplus\left(H_{0}^{1}\left(\,]-1/2,\infty[\,\right)\right)^{m_{+}}\oplus\left(H_{0}^{1}\left(\,]-\infty,1/2[\,\right)\right)^{m_{-}},
\end{align*}
where $H_{0}^{1}$ stands for the closure of smooth functions with
compact support in the space $H^{1}$; that is the space of Sobolev functions vanish at the boundary.  We now provide a congruence allowing us to transform the operator $\partial_x$ on $\bigoplus_{k=1}^N L^2 (I_k)$ to the standard space $L^2(\#N_{\mathrm{f}},\#M_+,\#M_-)$.
\begin{prop}
\label{prop:tranformation_d_dn}Let $a,b\in\mathbb{R},a<b$.

\begin{enumerate}[(a)]

\item Consider
\begin{align*}
\phi:\left]-1/2,1/2\right[ & \to\left]a,b\right[\\
x & \mapsto-\left(x-\frac{1}{2}\right)a+\left(x+\frac{1}{2}\right)b.
\end{align*}
Then $\phi$ is invertible\footnote{It is easily verified that
\begin{align*}
\phi^{-1}:\left]a,b\right[ & \to\left]-1/2,1/2\right[\\
x & \mapsto\frac{1}{b-a}x-\frac{b+a}{2\left(b-a\right)}
\end{align*}
is the inverse of $\phi$.} and 
\begin{align*}
\Phi:L^{2}\left(\left]a,b\right[\right) & \to L^{2}\left(\left]-1/2,1/2\right[\right)\\
u & \mapsto\sqrt{b-a}\left(u\circ\phi\right)
\end{align*}
is unitary.

\item Consider 
\begin{align*}
\phi:\left]-1/2,\infty\right[ & \to\left]a,\infty\right[\\
x & \mapsto x+a+\frac{1}{2}.
\end{align*}
Then
\begin{align*}
\Phi:L^{2}\left(\left]a,\infty\right[\right) & \to L^{2}\left(\left]-1/2,\infty\right[\right)\\
u & \mapsto u\circ\phi
\end{align*}
is unitary.

\item Let $P_{1}=P_{1}^{*}\in\mathbb{R}^{N\times N}$ be a diagonal
(hence, compatible) matrix , $n\coloneqq\#N_{\mathrm{f}},m_{+}\coloneqq\#M_{+},m_{-}\coloneqq\#M_{-}$
and
\begin{align*}
\partial_{x,\textnormal{n}}\colon H^{1}(n,m_{+},m_{-})\subseteq L^{2}(n,m_{+},m_{-}) & \to L^{2}(n,m_{+},m_{-})\\
\left(\phi_{k}\right)_{k} & \mapsto\left(\phi'_{k}\right)_{k}.
\end{align*}
Then there exists a unitary operator $\Psi\colon\bigoplus_{k=1}^{N}L^{2}(I_{k})\to L^{2}(n,m_{+},m_{-})$
and a diagonal, real matrix $\tilde{P}_{1}$ such that
\[
\Psi P_{1}\partial_{x}\Psi^{*}=\tilde{P}_{1}\partial_{x,\textnormal{n}}.
\]

\end{enumerate}
\end{prop}

\begin{proof}
(a) Let $u\in L^{2}(]a,b[)$. Then we compute
\begin{equation*}
\int_{-1/2}^{1/2}\left|\sqrt{b-a}\:u\left(\phi\left(x\right)\right)\right|^{2}dx=\int_{-1/2}^{1/2}\left|u\left(\phi\left(x\right)\right)\right|^{2}\left(b-a\right)dx  =\int_{a}^{b}\left|u\left(y\right)\right|^{2}dy.
\end{equation*}
Since $\Phi$ is also onto by the invertibility of $\phi$, the assertion follows.

(b) The assertion follows with an elementary computation similar to (a).

(c) Without loss of generality, we assume that $\{1,\ldots,n\}=N_{\textnormal{f}}$
and that $\{n+1,\ldots,n+m_{+}\}=M_{+}$ as well as $\{n+m_{+}+1,\ldots,N\}=M_{-}$.
For $k\in\{1,\ldots,n\}$ we find $a_{k},b_{k}\in\mathbb{R}$ such
that $I_{k}=]a_{k},b_{k}[$. Let $\phi_{k}$ be as in (a) with $a_{k},b_{k}$
replacing $a,b$. Then we have for all $k\in\{1,\ldots,n\}$
\[
\partial_{x}\left(u\circ\phi_{k}\right)=\left(b_{k}-a_{k}\right)\left(\partial_{x}u\right)\circ\phi_{k}.
\]
Hence, 
\[
\partial_{x,\textnormal{n}}\Phi_{k}=\left(b_{k}-a_{k}\right)\Phi_{k}\partial_{x},
\]
where we denoted $\Phi_{k}$ according to $\Phi$ as in (a) replacing
$\phi$ by $\phi_{k}$. Next, let $\Phi_{n+1},\ldots,\Phi_{m_{+}}$
be as in (b) with $a$ appropriately replaced in order that $\Phi_{k}\colon L^{2}(I_{k})\to L^{2}(]-1/2,\infty[)$
is unitary, $k\in\{n+1,\ldots,n+m_{+}\}$. For $k\in\{n+m_{+}+1,\ldots,N\}$,
we find, similar to (b), a unitary $\Phi_{k}\colon L^{2}(I_{k})\to L^{2}(]-\infty,1/2[).$
Thus, for all $k\in\{n+1,\ldots,N\},$we obtain
\[
\partial_{x,\textnormal{n}}\Phi_{k}=\Phi_{k}\partial_{x}.
\]
In consequence, denoting 
\[
\Psi\coloneqq\left(\begin{array}{cccccc}
\Phi_{1} & 0 &  & \cdots &  & 0\\
0 & \ddots\\
 &  & \Phi_{n} & \ddots &  & \vdots\\
\vdots &  & \ddots & \Phi_{n+1}\\
 &  &  &  & \ddots & 0\\
0 &  & \cdots &  & 0 & \Phi_{N}
\end{array}\right)\text{ and }\mathbf{b-a}\coloneqq\left(\begin{array}{cccccc}
b_{1}-a_{1} & 0 &  & \cdots &  & 0\\
0 & \ddots\\
 &  & b_{n}-a_{n} & \ddots &  & \vdots\\
\vdots &  & \ddots & 1\\
 &  &  &  & \ddots & 0\\
0 &  & \cdots &  & 0 & 1
\end{array}\right),
\]
we deduce that $\Psi$ is unitary and since $P_{1}$ is diagonal
\begin{align*}
\Psi P_{1}\partial_{x}\Psi^{*} & =\Psi P_{1}\left(\mathbf{b-a}\right)\Psi^{*}\partial_{x,\textnormal{n}}\\
 & =P_{1}\left(\mathbf{b-a}\right)\partial_{x,\textnormal{n}}\\
 & =\sqrt{\left(\mathbf{b-a}\right)}P_{1}\sqrt{\left(\mathbf{b-a}\right)}\partial_{x,\textnormal{n}}.
\end{align*}
Thus, the assertion follows with 
\[
\tilde{P}_{1}=\sqrt{\left(\mathbf{b-a}\right)}P_{1}\sqrt{\left(\mathbf{b-a}\right)}.\qedhere
\]
\end{proof}
\begin{rem}
\label{rem:BlockdiagonalP1}It can easily be seen from the proof that
the statement in (c) remains true, if $m_{+}=m_{-}=0$, $I_{k}=]a,b[$
for all $k\in\{1,\ldots,N\}$ and $P_{1}=P_{1}^{*}$ (so $P_{1}$
need not necessarily be diagonal). The matrix $\tilde{P}_{1}$ claimed
to exist then has the form
\[
\tilde{P}_{1}=\sqrt{\left(b-a\right)}P_{1}\sqrt{\left(b-a\right)}
\]
and is not necessarily diagonal either.
\end{rem}

If it is clear from the context, we will also write $\partial_{x}$
instead of $\partial_{x,\mathrm{n}}$ (as it is defined in Proposition
\ref{prop:tranformation_d_dn} (c)). 

We recall that by the Sobolev embedding theorem, 
\[
H^{1}(n,m_{+},m_{-})\subseteq\left(C\left([-1/2,1/2]\right)\right)^{n}\times\left(C_{0}\left(\,[-1/2,\infty[\,\right)\right)^{m_{+}}\times\left(C_{0}\left(\,]-\infty,1/2]\,\right)\right)^{m_{-}},
\]
where we denote by $C_0(I)$ for an interval $I\subseteq \R$ the closure of $C_c(I)$ (continuous functions with compact support) with repsect to the supremum-norm.  

This fact will be used frequently throughout the manuscript. 

A consequence of integration by parts and the Sobolev embedding theorem
is the next proposition. 
\begin{prop}
\label{prop:inner_product_d}Let $u=\left(\begin{array}{c}
u_{1/2}\\
u_{\infty}\\
u_{-\infty}
\end{array}\right),v=\left(\begin{array}{c}
v_{1/2}\\
v_{\infty}\\
v_{-\infty}
\end{array}\right)\in H^{1}(n,m_{+},m_{-}).$ Then we find 
\begin{align*}
\left\langle \partial_{x}u,v\right\rangle +\left\langle u,\partial_{x}v\right\rangle  & =\left\langle u_{1/2}\left(\frac{1}{2}-\right),v_{1/2}\left(\frac{1}{2}-\right)\right\rangle -\left\langle u_{1/2}\left(-\frac{1}{2}+\right),v_{1/2}\left(-\frac{1}{2}+\right)\right\rangle +\\
 & -\left\langle u_{\infty}\left(-\frac{1}{2}+\right),v_{\infty}\left(-\frac{1}{2}+\right)\right\rangle +\left\langle u_{-\infty}\left(\frac{1}{2}-\right),v_{-\infty}\left(\frac{1}{2}-\right)\right\rangle ,\\
 & =\left\langle \left(\begin{array}{c}
u_{1/2}\left(\frac{1}{2}-\right)\\
u_{-\infty}\left(\frac{1}{2}-\right)
\end{array}\right),\left(\begin{array}{c}
v_{1/2}\left(\frac{1}{2}-\right)\\
v_{-\infty}\left(\frac{1}{2}-\right)
\end{array}\right)\right\rangle -\left\langle \left(\begin{array}{c}
u_{1/2}\left(-\frac{1}{2}+\right)\\
u_{\infty}\left(-\frac{1}{2}+\right)
\end{array}\right),\left(\begin{array}{c}
v_{1/2}\left(-\frac{1}{2}+\right)\\
v_{\infty}\left(-\frac{1}{2}+\right)
\end{array}\right)\right\rangle .
\end{align*}
\end{prop}

Denote by $\mathring{\partial}_x$ the restriction of $\partial_x$ to $H_0^1(n,m_+,m_-)$. A straighforward consequence of the previous observation is the following.
\begin{prop}
\label{thm:adjoint_d}The operators $\partial_{x}$ and $\mathring{\partial}_{x}$
are densely defined and closed on $L^{2}(n,m_{+},m_{-})$. Moreover,
\[
\partial_{x}^{*}=-\mathring{\partial}_{x}\text{ and }-\mathring{\partial}_{x}^{*}=\partial_{x},
\]
where the adjoints are computed in $L^2(n,m_+,m_-)$. 
\end{prop}

The next statement summarises the description of all linear, dissipative
operator extensions of the minimal operator $\mathring{\partial}_{1}$ (see also \cite[p. 18-22]{Showalter1997} for related material).

\begin{thm}
\label{thm:char_accretivity}Let $\mathring{\partial}_{x}\subseteq D\subseteq\partial_{x}$
be a linear operator on $L^{2}(n,m_{+},m_{-})$. 

\begin{enumerate}[(a)]

\item Then the following statements are equivalent:

\begin{enumerate}[(i)]

\item $D$ is accretive; that is, 
\[
\Re\langle Du,u\rangle\geq0\quad(u\in\mathrm{dom}(D)).
\]

\item There exists $M\in\mathbb{R}^{\left(n+m_{+}\right)\times\left(n+m_{-}\right)}$
with $M^{*}M\leq1$ such that
\[
\mathrm{dom}(D)\subseteq\left\{ u=(u_{1/2},u_{\infty},u_{-\infty})\in H^{1}(n,m_{+},m_{-})\,;\,M\left(\begin{array}{c}
u_{1/2}\left(\frac{1}{2}-\right)\\
u_{-\infty}\left(\frac{1}{2}-\right)
\end{array}\right)+\left(\begin{array}{c}
u_{1/2}\left(-\frac{1}{2}+\right)\\
u_{\infty}\left(-\frac{1}{2}+\right)
\end{array}\right)=0\right\} .
\]

\end{enumerate}

\item Let $M\in\mathbb{R}^{\left(n+m_{+}\right)\times\left(n+m_{-}\right)}$
be such that 
\[
\mathrm{dom}(D)=\left\{ u=(u_{1/2},u_{\infty},u_{-\infty})\in H^{1}(n,m_{+},m_{-})\,;\,M\left(\begin{array}{c}
u_{1/2}\left(\frac{1}{2}-\right)\\
u_{-\infty}\left(\frac{1}{2}-\right)
\end{array}\right)+\left(\begin{array}{c}
u_{1/2}\left(-\frac{1}{2}+\right)\\
u_{\infty}\left(-\frac{1}{2}+\right)
\end{array}\right)=0\right\} .
\]
Then\footnote{Note that skew-selfadjointness of $D$ requires $M$ to be unitary and so
in particular $m_{+}=m_{-}$ .} $\mathring{\partial}_{x}\subseteq-D^{*}\subseteq\partial_{x}$ and
\[
\mathrm{dom}(D^{*})=\left\{ v=(v_{1/2},v_{\infty},v_{-\infty})\in H^{1}(n,m_{+},m_{-})\,;\,\left(\begin{array}{c}
v_{1/2}\left(\frac{1}{2}-\right)\\
v_{-\infty}\left(\frac{1}{2}-\right)
\end{array}\right)+M^{*}\left(\begin{array}{c}
v_{1/2}\left(-\frac{1}{2}+\right)\\
v_{\infty}\left(-\frac{1}{2}+\right)
\end{array}\right)=0\right\} .
\]

\item $D$ is \emph{maximal accretive}; that is, $D$ is accretive
and there exists no accretive relation extending $D$ (or, equivalently,
$D$ and $D^{\ast}$ are accretive) if and only if there exists $M\in\mathbb{R}^{\left(n+m_{+}\right)\times\left(n+m_{-}\right)}$
with $M^{*}M\leq1$ such that
\[
\mathrm{dom}(D)=\left\{ u=(u_{1/2},u_{\infty},u_{-\infty})\in H^{1}(n,m_{+},m_{-})\,;\,M\left(\begin{array}{c}
u_{1/2}\left(\frac{1}{2}-\right)\\
u_{-\infty}\left(\frac{1}{2}-\right)
\end{array}\right)+\left(\begin{array}{c}
u_{1/2}\left(-\frac{1}{2}+\right)\\
u_{\infty}\left(-\frac{1}{2}+\right)
\end{array}\right)=0\right\} .
\]

\end{enumerate}
\end{thm}

\begin{proof}
(a) Assume $D$ is accretive. Then, by Proposition \ref{prop:inner_product_d},
we deduce that, for all $u\in\mathrm{dom}(D)\subseteq H^{1}(n,m_{+},m_{-}),$
\[
0\leq2\Re\langle Du,u\rangle=\left\langle \left(\begin{array}{c}
u_{1/2}\left(\frac{1}{2}-\right)\\
u_{-\infty}\left(\frac{1}{2}-\right)
\end{array}\right),\left(\begin{array}{c}
u_{1/2}\left(\frac{1}{2}-\right)\\
u_{-\infty}\left(\frac{1}{2}-\right)
\end{array}\right)\right\rangle -\left\langle \left(\begin{array}{c}
u_{1/2}\left(-\frac{1}{2}+\right)\\
u_{\infty}\left(-\frac{1}{2}+\right)
\end{array}\right),\left(\begin{array}{c}
u_{1/2}\left(-\frac{1}{2}+\right)\\
u_{\infty}\left(-\frac{1}{2}+\right)
\end{array}\right)\right\rangle .
\]
Hence, 
\begin{equation}\label{eq:M}
\left\langle \left(\begin{array}{c}
u_{1/2}\left(-\frac{1}{2}+\right)\\
u_{\infty}\left(-\frac{1}{2}+\right)
\end{array}\right),\left(\begin{array}{c}
u_{1/2}\left(-\frac{1}{2}+\right)\\
u_{\infty}\left(-\frac{1}{2}+\right)
\end{array}\right)\right\rangle \leq\left\langle \left(\begin{array}{c}
u_{1/2}\left(\frac{1}{2}-\right)\\
u_{-\infty}\left(\frac{1}{2}-\right)
\end{array}\right),\left(\begin{array}{c}
u_{1/2}\left(\frac{1}{2}-\right)\\
u_{-\infty}\left(\frac{1}{2}-\right)
\end{array}\right)\right\rangle .
\end{equation}
Therefore, for all $\left(\begin{array}{c}
u_{1/2}\left(\frac{1}{2}-\right)\\
u_{-\infty}\left(\frac{1}{2}-\right)
\end{array}\right)$ with $u\in\mathrm{dom}(D)$, $M\left(\begin{array}{c}
u_{1/2}\left(\frac{1}{2}-\right)\\
u_{-\infty}\left(\frac{1}{2}-\right)
\end{array}\right)\coloneqq-\left(\begin{array}{c}
u_{1/2}\left(-\frac{1}{2}+\right)\\
u_{\infty}\left(-\frac{1}{2}+\right)
\end{array}\right)$ gives rise to a linear mapping defined on 
\[
R\coloneqq\left\{ \left(\begin{array}{c}
u_{1/2}\left(\frac{1}{2}-\right)\\
u_{-\infty}\left(\frac{1}{2}-\right)
\end{array}\right)\,;\,u\in\mathrm{dom}(D)\right\} \subseteq\mathbb{R}^{n+m_{-}}.
\]
 We put $M=0$ on $R^{\bot_{\mathbb{R}^{n+m_{-}}}}$. Thus, $M\in L(\mathbb{R}^{n+m_{-}},\mathbb{R}^{n+m_{+}})$
with $\|M\|\leq1$ by \prettyref{eq:M}. Hence, identifying $M$ with its matrix representation
$M\in\mathbb{R}^{\left(n+m_{+}\right)\times\left(n+m_{-}\right)}$
we obtain (ii). Next, assume (ii) and let $M$ be as in (ii). Then
we compute using Proposition \ref{prop:inner_product_d} for all $u\in\mathrm{dom}(D)$
\begin{align*}
2\Re\langle Du,u\rangle & =\left\langle \left(\begin{array}{c}
u_{1/2}\left(\frac{1}{2}-\right)\\
u_{-\infty}\left(\frac{1}{2}-\right)
\end{array}\right),\left(\begin{array}{c}
u_{1/2}\left(\frac{1}{2}-\right)\\
u_{-\infty}\left(\frac{1}{2}-\right)
\end{array}\right)\right\rangle -\left\langle \left(\begin{array}{c}
u_{1/2}\left(-\frac{1}{2}+\right)\\
u_{\infty}\left(-\frac{1}{2}+\right)
\end{array}\right),\left(\begin{array}{c}
u_{1/2}\left(-\frac{1}{2}+\right)\\
u_{\infty}\left(-\frac{1}{2}+\right)
\end{array}\right)\right\rangle \\
 & =\left\langle \left(\begin{array}{c}
u_{1/2}\left(\frac{1}{2}-\right)\\
u_{-\infty}\left(\frac{1}{2}-\right)
\end{array}\right),\left(\begin{array}{c}
u_{1/2}\left(\frac{1}{2}-\right)\\
u_{-\infty}\left(\frac{1}{2}-\right)
\end{array}\right)\right\rangle -\left\langle M\left(\begin{array}{c}
u_{1/2}\left(\frac{1}{2}-\right)\\
u_{-\infty}\left(\frac{1}{2}-\right)
\end{array}\right),M\left(\begin{array}{c}
u_{1/2}\left(\frac{1}{2}-\right)\\
u_{-\infty}\left(\frac{1}{2}-\right)
\end{array}\right)\right\rangle \\
 & \geq0
\end{align*}
as $\|M\|\leq 1$.

(b) It is clear that $H_{0}^{1}(n,m_{+},m_{-})\subseteq\mathrm{dom}(D^{*})\subseteq H^{1}(n,m_{+},m_{-}).$
Thus, for $u\in\mathrm{dom}(D)$ and $v\in H^{1}(n,m_{+},m_{-})$
we can use Proposition \ref{prop:inner_product_d} and deduce
\begin{align*}
\langle Du,v\rangle+\langle u,\partial_{x}v\rangle & =\left\langle \left(\begin{array}{c}
u_{1/2}\left(\frac{1}{2}-\right)\\
u_{-\infty}\left(\frac{1}{2}-\right)
\end{array}\right),\left(\begin{array}{c}
v_{1/2}\left(\frac{1}{2}-\right)\\
v_{-\infty}\left(\frac{1}{2}-\right)
\end{array}\right)\right\rangle -\left\langle \left(\begin{array}{c}
u_{1/2}\left(-\frac{1}{2}+\right)\\
u_{\infty}\left(-\frac{1}{2}+\right)
\end{array}\right),\left(\begin{array}{c}
v_{1/2}\left(-\frac{1}{2}+\right)\\
v_{\infty}\left(-\frac{1}{2}+\right)
\end{array}\right)\right\rangle \\
 & =\left\langle \left(\begin{array}{c}
u_{1/2}\left(\frac{1}{2}-\right)\\
u_{-\infty}\left(\frac{1}{2}-\right)
\end{array}\right),\left(\begin{array}{c}
v_{1/2}\left(\frac{1}{2}-\right)\\
v_{-\infty}\left(\frac{1}{2}-\right)
\end{array}\right)\right\rangle +\left\langle M\left(\begin{array}{c}
u_{1/2}\left(\frac{1}{2}-\right)\\
u_{-\infty}\left(\frac{1}{2}-\right)
\end{array}\right),\left(\begin{array}{c}
v_{1/2}\left(-\frac{1}{2}+\right)\\
v_{\infty}\left(-\frac{1}{2}+\right)
\end{array}\right)\right\rangle \\
 & =\left\langle \left(\begin{array}{c}
u_{1/2}\left(\frac{1}{2}-\right)\\
u_{-\infty}\left(\frac{1}{2}-\right)
\end{array}\right),\left(\begin{array}{c}
v_{1/2}\left(\frac{1}{2}-\right)\\
v_{-\infty}\left(\frac{1}{2}-\right)
\end{array}\right)+M^{*}\left(\begin{array}{c}
v_{1/2}\left(-\frac{1}{2}+\right)\\
v_{\infty}\left(-\frac{1}{2}+\right)
\end{array}\right)\right\rangle .
\end{align*}
Next, let $(x,y)\in\mathbb{R}^{n}\times\mathbb{R}^{m_{-}}.$ Then
$M(x,y)\in\mathbb{R}^{n}\times\mathbb{R}^{m_{+}}.$ Using suitable
piecewise linear functions, it is not difficult to construct $u\in H^{1}(n,m_{+},m_{-})$
such that 
\[
\left(\begin{array}{c}
u_{1/2}\left(\frac{1}{2}-\right)\\
u_{-\infty}\left(\frac{1}{2}-\right)
\end{array}\right)=\left(\begin{array}{c}
x\\
y
\end{array}\right)\text{ and }-M(x,y)=\left(\begin{array}{c}
u_{1/2}\left(-\frac{1}{2}+\right)\\
u_{\infty}\left(-\frac{1}{2}+\right)
\end{array}\right).
\]
Hence, 
\[
\left\{ \left(\begin{array}{c}
u_{1/2}\left(\frac{1}{2}-\right)\\
u_{-\infty}\left(\frac{1}{2}-\right)
\end{array}\right)\,;\,u\in\mathrm{dom}(D)\right\} =\mathbb{R}^{n}\times\mathbb{R}^{m_{-}}.
\]
As a consequence of this and the above computation, we deduce that
$v\in\mathrm{dom}(D^{*})$ if and only if $v\in H^{1}(n,m_{+},m_{-})$
and
\[
\left(\begin{array}{c}
v_{1/2}\left(\frac{1}{2}-\right)\\
v_{-\infty}\left(\frac{1}{2}-\right)
\end{array}\right)+M^{*}\left(\begin{array}{c}
v_{1/2}\left(-\frac{1}{2}+\right)\\
v_{\infty}\left(-\frac{1}{2}+\right)
\end{array}\right)=0,
\]
which establishes (b).

(c) At first we assume that $\mathrm{dom}(D)$ can be written as it
is given in (c). Then, by (a), $D$ is accretive. Moreover, by (b),
\[
\mathrm{dom}(D^{*})=\left\{ v=(v_{1/2},v_{\infty},v_{-\infty})\in H^{1}(n,m_{+},m_{-})\:;\:\left(\begin{array}{c}
v_{1/2}\left(\frac{1}{2}-\right)\\
v_{-\infty}\left(\frac{1}{2}-\right)
\end{array}\right)+M^{*}\left(\begin{array}{c}
v_{1/2}\left(-\frac{1}{2}+\right)\\
v_{\infty}\left(-\frac{1}{2}+\right)
\end{array}\right)=0\right\} ,
\]
which means that $D^{*}$ is accretive (note that with $M^{*}M\leq1$
we have $MM^{*}\leq1$ and also that $-\mathring{{\partial}}_x\subseteq D^{*}\subseteq-\partial_x$).
Since $D$ is closed and densely defined, it follows that $D$ is
maximal accretive. On the other hand, if $D$ is maximal accretive,
$D$ is accretive and therefore by (a), we find $M\in\mathbb{R}^{\left(n+m_{+}\right)\times\left(n+m_{-}\right)}$
with $M^{*}M\leq1$ such that 
\[
\mathrm{dom}(D)\subseteq\left\{ u=(u_{1/2},u_{\infty},u_{-\infty})\in H^{1}(n,m_{+},m_{-})\:;\:M\left(\begin{array}{c}
u_{1/2}\left(\frac{1}{2}-\right)\\
u_{-\infty}\left(\frac{1}{2}-\right)
\end{array}\right)+\left(\begin{array}{c}
u_{1/2}\left(-\frac{1}{2}+\right)\\
u_{\infty}\left(-\frac{1}{2}+\right)
\end{array}\right)=0\right\} .
\]
By the first part of the proof of (c), we have that $\partial_{x}$
restricted to the right-hand side of this inclusion is maximal accretive.
Hence, by the maximality of $D$, the inclusion is an equality, which
establishes the assertion.
\end{proof}
\begin{rem} 
The latter result is a special case of \cite[Theorem 2.1]{Waurick_Wegner_2020}.
\end{rem}

We remark here that the results in the section can be generalised also to infinite networks, that is one considers operators on $\oplus_{j\in J}L^2 (I_j)$ for an arbitrary index set $J$. Note, however, that in this case the mapping $\Psi$ in \prettyref{prop:tranformation_d_dn} (c) exists as a unitary mapping, only if $\inf_{j\in J} |I_j|>0$. For the general case we refer to \cite[Section 7]{Waurick_Wegner_2020}.

\section{The Congruence of $\partial_{x}$ and $P_{1}\partial_{x}\mathcal{H}$\label{sec:The-Congruence-of}}

Throughout, let $P_{1}\in\R^{N\times N}$ be a compatible selfadjoint
and invertible matrix and $\mathcal{H}:\R \to \R^{N\times N}$ measurable and bounded such that $\mathcal{H}(x)$ is compatible for each $x\in \R$ and $\mathcal{H}$ attains values in the symmetric matrices and is uniformly positive definite, i.e., there exists $c>0$ such that $\mathcal{H}(x)\geq c$ for all $x\in \R$. Moreover, we identify the function $\mathcal{H}$ with its induced multiplication operator on $\bigoplus_{k=1}^N L^2(I_k)$.\footnote{Note that it suffices to define $\mathcal{H}$ on $\bigcup_{k=1}^N I_k$. However, such a function can easily be extended to $\mathbb{R}$ by setting $\mathcal{H}(x)=I_N$ for $x\notin \bigcup_{k=1}^N I_k$.}  The aim of the present section is to identify
the operator $P_{1}\partial_{x}\mathcal{H}$ on a suitable Hilbert space and the operator $\partial_{x}$
on $L^{2}(n,m_{+},m_{-})$ as mutually congruent operators, if we
choose $n,m_{+},m_{-}\in\N$ suitably. The reduction to the case $\mathcal{H}=1$ is standard and well-known in the theory of port-Hamiltonian systems (see e.g. \cite[Lemma 7.2.3]{Jacob_Zwart_ISEM}).

\begin{prop}\label{prop:Hamiltonian_congruence}
 Let $H\coloneqq \bigoplus_{k=1}^N L^2(I_k)$ equipped with the inner product
 \[
  \langle u,v\rangle_H\coloneqq \langle \mathcal{H}u,  v\rangle \quad (u,v\in H).
 \]
 Consider the mapping $B: H\to \bigoplus_{k=1}^N L^2(I_k)$ given by $Bu\coloneqq \mathcal{H}u$. Then 
 \[
  B^\ast : \bigoplus_{k=1}^N L^2(I_k)\to H,\quad v\mapsto v,
 \]
where the adjoint is computed with repsect to the inner products on $H$ and $\bigoplus_{k=1}^N L^2(I_k)$.
\end{prop}

\begin{proof}
 First note that the inner product on $H$ is well-defined due to the selfadjointness, boundedness and positive definiteness of $\mathcal{H}$. Moreover, $B$ is obviously linear and bounded and for $u\in H, v\in \bigoplus_{k=1}^N L^2(I_k)$ we compute
 \begin{equation*}
  \langle Bu,v\rangle =\langle \mathcal{H}u,v\rangle =\langle u,v\rangle_H,
 \end{equation*}
which shows the asserted formula for $B^\ast$.
\end{proof}

The latter proposition shows that $P_1\partial_x \mathcal{H}$ and $P_1\partial_x$ are congruent as operators on $H$ and $\bigoplus_{k=1}^N L^2(I_k)$, respectively. Next we show that $P_1\partial_x$ and $\partial_x$ are congruent as well. The strategy is as follows.
First, we diagonalise $P_{1}$ and then apply our congruence result
\ref{prop:tranformation_d_dn} (c). After that, we use the reflection
operators $\sigma_{-1}$ given by $(\sigma_{-1}u)(x)\coloneqq u(-x)$
to obtain a congruent operator of the form $D\partial_{x}$ on $L^{2}(n,m_{+},m_{-})$,
where $D$ is a diagonal matrix with positive diagonal entries. Finally,
using $\sqrt{D^{-1}}$ we obtain the asserted congruence to $\partial_{x}.$
The precise statement is as follows. %

\begin{thm}
\label{thm:main_congruence}Let $P_{1}=P_{1}^{*}\in\mathbb{R}^{N\times N}$
invertible and compatible. Then there exists $n,m_{+},m_{-}\in\mathbb{N}_{0}$
with $n+m_{+}+m_{-}=N$ and an invertible operator $\mathcal{V}\colon\bigoplus_{k\in\{1,\ldots,N\}}L^{2}(I_{k})\to L^{2}(n,m_{+},m_{-})$
such that
\[
\mathcal{V}P_{1}\partial_{x}\mathcal{V^{*}}=\partial_{x}.
\]
More precisely, $\mathcal{V}$ is given as a product of constant matrices,
the operator $\Psi$ given in Proposition \ref{prop:tranformation_d_dn}
(c) and a diagonal operator consiting of identities and reflections
on the diagonal.
\end{thm}

\begin{proof}
Since $P_{1}$ is selfadjoint, we find a unitary matrix $K\in\R^{N\times N}$
such that 
\[
KP_{1}K^{\ast}=D_{1},
\]
where $D_{1}\in\R^{N\times N}$ is a diagonal matrix, whose diagonal
entries are nonzero thanks to the invertibility of $P_{1}.$ By Proposition
\ref{prop:tranformation_d_dn} (c), we find a unitary mapping $\Psi\colon\bigoplus_{k\in\{1,\ldots,N\}}L^{2}(I_{k})\to L^{2}(n,\tilde{m}_{+},\tilde{m}_{-})$
with suitable $n,\tilde{m}_{+},\tilde{m}_{-}\in\mathbb{N}_{0}$ such
that $n+\tilde{m}_{+}+\tilde{m}_{-}=N$ and a diagonal matrix with
real, non-zero entries $\tilde{D}_{1}$ such that
\begin{align*}
\Psi D_{1}\partial_{x}\Psi^{*} & =\tilde{D}_{1}\partial_{x}.
\end{align*}

Then, $\tilde{D}_{1}$ can be written according to the latter block
decomposition with diagonal matrices $\tilde{D}_{1,\ell}\in\mathbb{R}^{\ell\times\ell}$
for $\ell\in\{n,\tilde{m}_{+},\tilde{m}_{-}\}$
\[
\left(\begin{array}{ccc}
\tilde{D}_{1,n} & 0 & 0\\
0 & \tilde{D}_{1,\tilde{m}_{+}} & 0\\
0 & 0 & \tilde{D}_{1,\tilde{m}_{-}}
\end{array}\right).
\]
Let $n^{+}$ be the number of positive numbers in the diagonal of
$\tilde{D}_{1,n},$ $n^{-}\coloneqq n-n^{+}$. Similarly we define
$\tilde{m}_{\pm}^{\pm}$ . Next, we define an operator $W:\bigoplus_{k=1}^{N}L^{2}(\R)\to\bigoplus_{k=1}^{N}L^{2}(\R)$
acting coordinate wise by 
\begin{align*}
\left(W_{k}u\right)(x) & \coloneqq\begin{cases}
u(x) & \text{if }\tilde{D}_{1,kk}>0,\\
u(-x) & \text{if }\tilde{D}_{1,kk}<0
\end{cases}\quad(x\in\R)
\end{align*}
for each $u\in L^{2}(\R)$ and $k\in\{1,\ldots,N\}.$ Moreover, it
is clear that there exists a permutation matrix $Q\in\R^{N\times N}$
such that $QW[L^{2}(n,\tilde{m}_{+},\tilde{m}_{-})]\subseteq L^{2}(n,m_{+},m_{-}),$
where $m_{+}\coloneqq\tilde{m}_{+}^{+}+\tilde{m}_{-}^{-},\:m_{-}\coloneqq\tilde{m}_{+}^{-}+\tilde{m}_{-}^{+}$.
Setting $V\coloneqq QW:L^{2}(n,\tilde{m}_{+},\tilde{m}_{-})\to L^{2}(n,m_{+},m_{-})$
we infer
\[
V\left(\tilde{D}_{1}\partial_{x}\right)V^{\ast}=D_{2}\partial_{x}
\]
on the Hilbert space $L^{2}(n,m_{+},m_{-})$ where $D_{2}\in\R^{N\times N}$
is a diagonal matrix with strictly positive diagonal entries. With
these transformation at hand, we obtain 
\begin{align*}
\partial_{x} & =\left(\sqrt{D_{2}}\right)^{-1}D_{2}\partial_{x}\left(\sqrt{D_{2}}\right)^{-1}\\
 & =\left(\sqrt{D_{2}}\right)^{-1}V\left(\tilde{D}_{1}\partial_{x}\right)V^{\ast}\left(\sqrt{D_{2}}\right)^{-1}\\
 & =\left(\sqrt{D_{2}}\right)^{-1}V\Psi(D_{1}\partial_{x})\Psi^{*}V^{\ast}\left(\sqrt{D_{2}}\right)^{-1}\\
 & =\left(\sqrt{D_{2}}\right)^{-1}V\Psi K(P_{1}\partial_{x})K^{\ast}\Psi^{*}V^{\ast}\left(\sqrt{D_{2}}\right)^{-1},
\end{align*}
thus the assertion follows with $\mathcal{V\coloneqq}\left(\sqrt{D_{2}}\right)^{-1}V\Psi K.$
\end{proof}
\begin{rem}[Boundary Conditions for $P_{1}\partial_{x}$]
 Note that a closer inspection of the proof of Theorem \ref{thm:main_congruence}
reveals that also 
\[
\mathcal{V}P_{1}\mathring{\partial}_{1}\mathcal{V}^{*}=\mathring{\partial}_{1},
\]
where the $\ \mathring{}\negmedspace$~ stands for the realisation
of $\partial_{x}$ restricted to $\bigoplus_{k\in\{1,\ldots,N\}}H_{0}^{1}(I_{k})$
in the former and to $H_{0}^{1}(n,m_{+},m_{-})$ in the latter case.
The reason for this is that the transformation $\mathcal{V}$ maps
smooth compactly supported functions into smooth compactly supported
functions. In particular, this means that any choice of linear boundary
conditions for $P_{1}\partial_{x}$ is in one-to-one correspondence
to a boundary condition for $\partial_{x}$. We have classified all
boundary conditions for $\partial_{x}$ leading to an m-accretive
operator realisation of $\partial_{x}$. Hence, we obtain a complete
description of all boundary conditions for $P_{1}\partial_{x}$ leading
to m-accretive operator realisations as a consequence of Theorem \ref{thm:main_congruence}.
We shall see in Lemma \ref{lem:transformation_bd_value} below that
the particular form of the congruence is not important. In fact, it
turns out that boundary conditions for $P_{1}\partial_{x}$ can be
rephrased into boundary conditions for $\partial_{x}$ as long as
$\mathcal{V}P_{1}\partial_{x}\mathcal{V}^{*}=\partial_{x}$ for \emph{any}
invertible operator $\mathcal{V}.$ 
\end{rem}

We conclude this section by looking at a variant of the above result
under the assumption 
\begin{equation}
-\infty<\inf_{k\in\{1,\ldots,N\}}I_{k}<\sup_{k\in\{1,\ldots,N\}}I_{k}<\infty\label{eq:finiteInterv}
\end{equation}
of finiteness of all intervals. Then $n=N,m_{+}=m_{-}=0$ in Theorem
\ref{thm:main_congruence}. We first provide another representation
of $\partial_{x}$ on $L^{2}(N,0,0)=L^{2}(]-1/2,1/2[)^{N}$. It will
be instrumental of the proof that the derivative of an odd function
is even and that the derivative of an even function is odd.

\begin{thm}[{\cite[p. 61-62]{Picard2013_mother} and \cite[p. 2811-2812]{Trostorff2014}}]
 Let $\partial_{x}\colon H^{1}(]-1/2,1/2[)^{N}\subseteq L^{2}(]-1/2,1/2[)^{N}\to L^{2}(]-1/2,1/2[)^{N}.$
We define 
\begin{align*}
\iota_{\textnormal{e}}\colon L_{\textnormal{e}}^{2}(]-1/2,1/2[)^{N} & \to L^{2}(]-1/2,1/2[)^{N}\\
\iota_{\textnormal{o}}\colon L_{\textnormal{o}}^{2}(]-1/2,1/2[)^{N} & \to L^{2}(]-1/2,1/2[)^{N}
\end{align*}
the canonical embeddings from the (component-wise) even and odd functions
on $L^{2}(]-1/2,1/2[)$ into $L^{2}(]-1/2,1/2[)$. Then
\[
\left(\begin{array}{c}
\iota_{\textnormal{e}}^{*}\\
\iota_{\textnormal{o}}^{*}
\end{array}\right)\partial_{x}\left(\begin{array}{cc}
\iota_{\textnormal{e}} & \iota_{\textnormal{o}}\end{array}\right)=\left(\begin{array}{cc}
0 & \partial_{x,\textnormal{o}}\\
\partial_{x,\textnormal{e}} & 0
\end{array}\right),
\]
where 
\begin{align*}
\partial_{x,\textnormal{e}}\colon H_{\textnormal{e}}^{1}(]-1/2,1/2[)^{N}\subseteq L_{\textnormal{e}}^{2}(]-1/2,1/2[)^{N} & \to L_{\textnormal{o}}^{2}(]-1/2,1/2[)^{N}\\
\phi & \mapsto\phi'
\end{align*}
and $H_{\textnormal{e}}^{1}(]-1/2,1/2[)^{N}\coloneqq H^{1}(]-1/2,1/2[)^{N}\cap L_{\textnormal{e}}^{2}(]-1/2,1/2[)^{N}$;
similarly $\partial_{x,\textnormal{o}}$ and $H_{\textnormal{o}}^{1}(]-1/2,1/2[)^{N}$. 
\end{thm}

\begin{rem}
Note that the previous theorem permits to compute the form of the
boundary conditions for $\mathring{\partial}_{x}\subseteq A\subseteq\partial_{x}$
maximal accretive in terms of restrictions and extensions of $\left(\begin{array}{cc}
0 & \partial_{x,\textnormal{o}}\\
\partial_{x,\textnormal{e}} & 0
\end{array}\right)$. Indeed, let $\mathring{\partial}_{x}\subseteq A\subseteq\partial_{x}$
be maximal accretive and let $M\in\mathbb{R}^{N\times N}$ be such
that 
\[
\mathrm{\mathrm{dom}}(A)=\{u\in H^{1}(]-1/2,1/2[)^{N}\,;\,Mu(\frac{1}{2}-)+u(-\frac{1}{2}+)=0\}.
\]
Note that a small computation (and invariance of smooth compactly
supported functions) shows that 
\[
\left(\begin{array}{c}
\iota_{\textnormal{e}}^{*}\\
\iota_{\textnormal{o}}^{*}
\end{array}\right)\mathring{\partial}_{x}\left(\begin{array}{cc}
\iota_{\textnormal{e}} & \iota_{\textnormal{o}}\end{array}\right)=\left(\begin{array}{cc}
0 & \mathring{\partial}_{x,\textnormal{o}}\\
\mathring{\partial}_{x,\textnormal{e}} & 0
\end{array}\right).
\]
Also note that 
\begin{align*}
u(\frac{1}{2}-) & =\left(\iota_{\textnormal{e}}\iota_{\textnormal{e}}^\ast u+\iota_{\textnormal{o}}\iota_{\textnormal{o}}^\ast u\right)(\frac{1}{2}-)=\left(\iota_{\textnormal{e}}\iota_{\textnormal{e}}^\ast u\right)(\frac{1}{2}-)+\left(\iota_{\textnormal{o}}\iota_{\textnormal{o}}^\ast u\right)(\frac{1}{2}-),\\
u(-\frac{1}{2}+) & =\left(\iota_{\textnormal{e}}\iota_{\textnormal{e}}^\ast u+\iota_{\textnormal{o}}\iota_{\textnormal{o}}^\ast u\right)(-\frac{1}{2}+)=\left(\iota_{\textnormal{e}}\iota_{\textnormal{e}}^\ast u\right)(\frac{1}{2}-)-\left(\iota_{\textnormal{o}}\iota_{\textnormal{o}}^\ast u\right)(\frac{1}{2}-).
\end{align*}
Hence, 
\[
\mathrm{\mathrm{dom}}(A)=\left\{ u\in H^{1}(]-1/2,1/2[)^{N};M\left(\left(\iota_{\textnormal{e}}\iota_{\textnormal{e}}^\ast u\right)(\frac{1}{2}-)+\left(\iota_{\textnormal{o}}\iota_{\textnormal{o}}^\ast u\right)(\frac{1}{2}-)\right)+\left(\iota_{\textnormal{e}}\iota_{\textnormal{e}}^\ast u\right)(\frac{1}{2}-)-\left(\iota_{\textnormal{o}}\iota_{\textnormal{o}}^\ast u\right)(\frac{1}{2}-)=0\right\} .
\]
Thus, $\tilde{A}\coloneqq\left(\begin{array}{c}
\iota_{\textnormal{e}}^{*}\\
\iota_{\textnormal{o}}^{*}
\end{array}\right)A\left(\begin{array}{cc}
\iota_{\textnormal{e}} & \iota_{\textnormal{o}}\end{array}\right)$ is a m-accretive restriction of $\left(\begin{array}{cc}
0 & \partial_{x,\textnormal{o}}\\
\partial_{x,\textnormal{e}} & 0
\end{array}\right)$ with domain
\[
\mathrm{dom}(\tilde{A})=\left\{ \left(\begin{array}{c}
u_{\textnormal{e}}\\
u_{\textnormal{o}}
\end{array}\right)\in H_{\textnormal{e}}^{1}(]-1/2,1/2[)^{N}\oplus H_{\textnormal{o}}^{1}(]-1/2,1/2[)^{N};\left(\begin{array}{cc}
M+I & M-I\end{array}\right)\left(\begin{array}{c}
u_{\textnormal{e}}(1/2-)\\
u_{\textnormal{o}}(1/2-)
\end{array}\right)=0\right\} .
\]
This leads to an alternative proof of one equivalence in \cite[Theorem 7.2.4]{Jacob_Zwart_ISEM}
(pay also attention to \cite[Lemma 7.3.1]{Jacob_Zwart_ISEM}).
\end{rem}

\section{Well-posedness of port-Hamiltonian Boundary Control Systems}\label{sec:well-posedness}

In this section, we shall have a closer look at port-Hamiltonian systems in the simplier case $I_1=\ldots=I_N=]a,b[$ for some $a,b\in \R$. Note that then $n=N,m_+=m_-=0$ and $L^2(N,0,0)=L^2(]-1/2,1/2[)^N=L^2(]-1/2,1/2[;\R^N)$. 
In fact, because of the structure theorem, which allows us to represent any
port-Hamiltonian system as $\partial_{x}$, we are in a position to
describe a rather general class of port-Hamiltonian systems and discuss
related issues like well-posedness of associated boundary control
systems. The foundation
of this lies in the solution theory for evolutionary equations, see
e.g. \cite{Picard2009,Picard_McGhee2011,PTW2015_survey,ISEM2020}.

We briefly recall the setting of \cite[Section 5]{Jacob_Zwart_2018}
(see also \cite[Section 7]{Jacob_Zwart_ISEM}). For this, let $a,b\in\mathbb{R}$,
$a<b$, $P_{1}=P_{1}^{*}\in\mathbb{R}^{N\times N}$ invertible, $P_{0}=-P_{0}^{*}\in\mathbb{R}^{N\times N},$
$\mathcal{H}\in L^{\infty}([a,b];\mathbb{R}^{N\times N})$. Assume
there exists $m,M\in\mathbb{R}_{>0}$ such that 
\[
mI_{N\times N}\leq\mathcal{H}(\zeta)=\mathcal{H}(\zeta)^{*}\leq MI_{N\times N}\quad(\text{a.e. }\zeta\in[a,b]).
\]
Let
\begin{align*}
A\colon\mathrm{dom}(A)\subseteq L_{\mathcal{H}}^{2}(]a,b[)^{N} & \to L_{\mathcal{H}}^{2}(]a,b[)^{N}\\
x & \mapsto P_{1}\left(\mathcal{H}x\right)'+P_{0}\mathcal{H}x,
\end{align*}
where
\begin{equation*}
\dom(\mathring{\partial}_x\mathcal{H})\subseteq \mathrm{dom}(A) \subseteq \dom(\partial_x \mathcal{H})   \quad \text{ and } \quad
L_{\mathcal{H}}^{2}(]a,b[)^{N}  \coloneqq(L^{2}(]a,b[)^{N};\langle\cdot,\mathcal{H}\cdot\rangle_{L^{2}(]a,b[)^{N}}).
\end{equation*}
In order to have a meaningful notion of well-posedness, classically,
people focus on the the generator propertiesof $A$. Generating a ($C_{0}$-)semigroup
of contractions can be characterised as the closed densely defined operator being  m-dissipative,  by
the Lumer--Philipps Theorem. In the particular case of port-Hamiltonian
systems this characterisation can be reformulated in terms of the
boundary conditions parametrised by some matrix $W_B$, see e.g.~\cite[Theorem 1]{Jacob_Morris_Zwart2015}
or \cite[Theorem 5.8]{Jacob_Zwart_2018}: 
\begin{thm}
\label{thm:well_posedness}The following conditions are equivalent:

\begin{enumerate}[(i)]

\item $-A$ generates a semigroup of contractions on $L_{\mathcal{H}}^{2}(]a,b[)^{N}$;
that is, $A$ is maximal accretive;

\item $A$ is accretive and there exists $W_B\in \R^{N\times 2N}$ such that 
\[
 \dom(A)=\left\{x\in L^2(]a,b[)^N\,;\,  \mathcal{H}x\in H^1(]a,b[)^N,\, W_B \begin{pmatrix}
                                                                             (\mathcal{H}x)(b)\\
                                                                             (\mathcal{H}x)(a)
                                                                            \end{pmatrix}=0 \right\}
\]

\item there exists $W_{B}\in \R^{N\times 2N}$ such that $\dom(A)$ is given as in (ii) and $W_B$  has rank $N$ and, in the sense of positive definiteness,
\[
W_{B}\left(\begin{array}{cc}
-P_{1} & P_{1}\\
I_{N\times N} & I_{N\times N}
\end{array}\right)^{-1}\left(\begin{array}{cc}
0 & I_{N\times N}\\
I_{N\times N} & 0
\end{array}\right)\left(W_{B}\left(\begin{array}{cc}
-P_{1} & P_{1}\\
I_{N\times N} & I_{N\times N}
\end{array}\right)^{-1}\right)^{*}\geq0.
\]

\end{enumerate}

If any of the above holds, then there exists an invertible matrix $C\in\R^{2N\times2N}$
and matrices $M,L\in\R^{N\times N}$ with $M^{\ast}M\leq1$ and $L$
invertible such that 
\[
W_{B}C=L\left(\begin{array}{cc}
M & 1\end{array}\right).
\]
\end{thm}

We aim to prove this theorem with the help of our congruence result,
Theorem \ref{thm:main_congruence}, and the characteristaion result,
Theorem \ref{thm:char_accretivity}. For this, we need to inspect
how the transformation $\mathcal{V}$ of Theorem \ref{thm:main_congruence}
acts on the boundary values of a function $u\in H^{1}([a,b])^{N}.$
It is remarkable that no knowledge of how $\mathcal{V}$ acts on functions
with vanishing boundary values is needed in order to obtain
that $H_{0}^{1}([a,b])^{N}$-functions are mapped onto $H_{0}^{1}([-1/2,1/2])^{N}$-functions. 
\begin{lem}
\label{lem:transformation_bd_value} Let $P_{1}=P_{1}^{\ast}\in\R^{N\times N}$
invertible. Moreover, let $\mathcal{V}:L^{2}(]a,b[)^{N}\to L^{2}(]-1/2,1/2[)^{N}$
be invertible such that 
\[
\mathcal{V}P_{1}\partial_{x}\mathcal{V}^{\ast}=\partial_{x}.
\]
Then for $u\in H^{1}(]-1/2,1/2[)^{N}$ we have that $\mathcal{V}^{\ast}u\in H^{1}(]a,b[)^{N}$
and there exists an invertible matrix $C\in\R^{2N\times2N}$ such
that 
\[
\left(\begin{array}{c}
(\mathcal{V}^{\ast}u)(b)\\
(\mathcal{V^{\ast}}u)(a)
\end{array}\right)=C\left(\begin{array}{c}
u(1/2)\\
u(-1/2)
\end{array}\right)\quad(u\in H^{1}(]-1/2,1/2[)^{N}).
\]

Moreover, $C$ satisfies
\[
C^{\ast}\left(\begin{array}{cc}
P_{1} & 0\\
0 & -P_{1}
\end{array}\right)C=\left(\begin{array}{cc}
1 & 0\\
0 & -1
\end{array}\right).
\]
\end{lem}

\begin{proof}
From the congruence we see that $\mathcal{V}^{\ast}u\in H^{1}(]a,b[)^{N}$
if $u\in H^{1}(]-1/2,1/2[)^{N}.$ Moreover, for $u,w\in H^{1}(]-1/2,1/2[)^{N}$
we compute 
\begin{align}
\left\langle \left(\begin{array}{c}
(\mathcal{V}^{\ast}u)(b)\\
(\mathcal{V^{\ast}}u)(a)
\end{array}\right),\left(\begin{array}{c}
P_{1}(\mathcal{V}^{\ast}w)(b)\\
-P_{1}(\mathcal{V}^{\ast}w)(a)
\end{array}\right)\right\rangle  & =\langle(\mathcal{V}^{\ast}u)(b),P_{1}(\mathcal{V}^{\ast}w)(b)\rangle-\langle(\mathcal{V}^{\ast}u)(a),P_{1}(\mathcal{V}^{\ast}w)(a)\rangle\nonumber \\
 & =\langle\partial_{x}\mathcal{V}^{\ast}u,P_{1}\mathcal{V}^{\ast}w\rangle+\langle\mathcal{V}^{\ast}u,\partial_{x}P_{1}\mathcal{V}^{\ast}w\rangle\nonumber \\
 & =\langle\mathcal{V}P_{1}\partial_{x}\mathcal{V}^{\ast}u,w\rangle+\langle u,\mathcal{V}P_{1}\partial_{x}\mathcal{V}^{\ast}w\rangle\nonumber \\
 & =\langle\partial_{x}u,w\rangle+\langle u,\partial_{x}w\rangle\nonumber \\
 & =\left\langle \left(\begin{array}{c}
u(1/2)\\
u(-1/2)
\end{array}\right),\left(\begin{array}{c}
w(1/2)\\
-w(-1/2)
\end{array}\right)\right\rangle .\label{eq:bd_value}
\end{align}
We consider now the binary relation 
\[
C\coloneqq\left\{ \left(\left(\begin{array}{c}
u(1/2)\\
u(-1/2)
\end{array}\right),\left(\begin{array}{c}
(\mathcal{V}^{\ast}u)(b)\\
(\mathcal{V^{\ast}}u)(a)
\end{array}\right)\right)\,;\,u\in H^{1}(]-1/2,1/2[)^{N}\right\} \subseteq\R^{2N}\times\R^{2N}.
\]
Obviously, $C$ is linear. Moreover, $C$ is a mapping. Indeed, due
to the linearity, it suffices to prove that $\left(\begin{array}{c}
u(1/2)\\
u(-1/2)
\end{array}\right)=0$ implies $\left(\begin{array}{c}
(\mathcal{V}^{\ast}u)(b)\\
(\mathcal{V^{\ast}}u)(a)
\end{array}\right)=0$. So, let $u\in H^{1}(]-1/2,1/2[)^{N}$ with $\left(\begin{array}{c}
u(1/2)\\
u(-1/2)
\end{array}\right)=0$. By \eqref{eq:bd_value} it follows that 
\[
\left\langle \left(\begin{array}{c}
(\mathcal{V}^{\ast}u)(b)\\
(\mathcal{V^{\ast}}u)(a)
\end{array}\right),\left(\begin{array}{c}
P_{1}(\mathcal{V}^{\ast}w)(b)\\
-P_{1}(\mathcal{V}^{\ast}w)(a)
\end{array}\right)\right\rangle =0\quad(w\in H^{1}(]-1/2,1/2[)^{N}).
\]
Let now $x,y\in\R^{N}$ and define $f(t)\coloneqq\frac{1}{b-a}\left((t-a)P_{1}^{-1}x+(t-b)P_{1}^{-1}y\right)$
for $t\in[a,b].$ Then $f\in H^{1}(]a,b[)^{N}$ and we set $w\coloneqq\left(\mathcal{V}^{\ast}\right)^{-1}f\in H^{1}(]-1/2,1/2[)^{N}$
(this follows from $\mathcal{V}P_{1}\partial_{x}\mathcal{V}^{\ast}=\partial_{x}$).
Then clearly, $P_{1}(\mathcal{V}^{\ast}w)(b)=x$ and $-P_{1}(\mathcal{V}^{\ast}w)(a)=y$
and so, we infer 
\[
\left\langle \left(\begin{array}{c}
(\mathcal{V}^{\ast}u)(b)\\
(\mathcal{V^{\ast}}u)(a)
\end{array}\right),\left(\begin{array}{c}
x\\
y
\end{array}\right)\right\rangle =0.
\]
Since this holds for all $x,y\in\R^{N},$ we derive $\left(\begin{array}{c}
(\mathcal{V}^{\ast}u)(b)\\
(\mathcal{V^{\ast}}u)(a)
\end{array}\right)=0$. In the same way one shows that $C$ is one-to-one. Moreover, it
is obvious that the domain of $C$ is $\R^{2N}$ (see also Lemma \ref{lem:surjectivity_trace})
and so, $C$ can be represented by an invertible matrix, which we
again denote by $C\in\R^{2N\times2N}.$ Thus, we have 
\[
\left(\begin{array}{c}
(\mathcal{V}^{\ast}u)(b)\\
(\mathcal{V^{\ast}}u)(a)
\end{array}\right)=C\left(\begin{array}{c}
u(1/2)\\
u(-1/2)
\end{array}\right)\quad(u\in H^{1}(]-1/2,1/2[)^{N}).
\]
Plugging in this representation in (\ref{eq:bd_value}), we obtain
\[
\left\langle C\left(\begin{array}{c}
u(1/2)\\
u(-1/2)
\end{array}\right),\left(\begin{array}{cc}
P_{1} & 0\\
0 & -P_{1}
\end{array}\right)C\left(\begin{array}{c}
w(1/2)\\
w(-1/2)
\end{array}\right)\right\rangle =\left\langle \left(\begin{array}{c}
u(1/2)\\
u(-1/2)
\end{array}\right),\left(\begin{array}{cc}
1 & 0\\
0 & -1
\end{array}\right)\left(\begin{array}{c}
w(1/2)\\
w(-1/2)
\end{array}\right)\right\rangle 
\]
for all $u,w\in H^{1}(]-1/2,1/2[)^{N}$ form which we infer 
\[
C^{\ast}\left(\begin{array}{cc}
P_{1} & 0\\
0 & -P_{1}
\end{array}\right)C=\left(\begin{array}{cc}
1 & 0\\
0 & -1
\end{array}\right).\qedhere
\]
\end{proof}
Next, we present our alternative way of characterising maximal accretivity
of $A$ in terms of the boundary conditions. Our approach makes use
of the similarity result presented in the previous section.
\begin{proof}[Proof of Theorem \ref{thm:well_posedness}]
 First, we note that we can assume without loss of generality that
$P_{0}=0$ (note that $P_0\mathcal{H}$ is skew-selfadjoint on $L^2_\mathcal{H}(]a,b[)^N$ and thus, does not effect the maximal accretivity of $A$) and $\mathcal{H}=1$, by \prettyref{prop:Hamiltonian_congruence}. 

Assume that $A$ is maximal accretive. We need to show that there exists $W_B\in \R^{N\times 2N}$ with 
\[
 \dom(A)=\left\{x\in L^2(]a,b[)^N\,;\,  \mathcal{H}x\in H^1(]a,b[)^N,\, W_B \begin{pmatrix}
                                                                             (\mathcal{H}x)(b)\\
                                                                             (\mathcal{H}x)(a)
                                                                            \end{pmatrix}=0 \right\}.
                                                                            \]
By \prettyref{thm:main_congruence} we find $\mathcal{V}:L^2(]a,b[)^N\to L^2(]-1/2,1/2[)^N$ invertible with $\mathcal{V}P_1\partial_x\mathcal{V}^\ast=\partial_x$. Hence, $D\coloneqq \mathcal{V} A\mathcal{V}^\ast$ is maximal accretive and satisfies $\mathring{\partial}_x\subseteq D\subseteq \partial_x$ (note that $\mathcal{V}$ leaves $H_0^1(]a,b[)^N$ invariant by \prettyref{lem:transformation_bd_value}). By \prettyref{thm:char_accretivity} we find a matrix $M\in \R^{N\times N}$ such that
\[
 \dom(D)=\{ v\in H^1(]a,b[)^N\,;\, Mv(1/2-)+v(-1/2+)=0\}
\]
which in turn implies 
\begin{align*}
 \dom(A)&=\{ u\in L^2(]a,b[)^N\,;\, (\mathcal{V}^\ast)^{-1}u \in \dom(D)\}\\
        &=\{ u\in H^1(]a,b[)^N\,;\, M (\mathcal{V}^\ast)^{-1}u (1/2-)+(\mathcal{V}^\ast)^{-1}u (-1/2+)=0 \} \\
        &= \left\{ u\in H^1(]a,b[)^N\,;\, \begin{pmatrix}
                                           M & 1
                                           \end{pmatrix}\begin{pmatrix}
                                                         (V^\ast)^{-1}u(1/2)\\
                                                         (V^\ast)^{-1}u(-1/2)
                                                       \end{pmatrix}=0 \right\} \\
        &= \left\{ u\in H^1(]a,b[)^N\,;\, \begin{pmatrix}
                                           M & 1
                                           \end{pmatrix}C^{-1} \begin{pmatrix}
                                                         u(b)\\
                                                         u(a)
                                                       \end{pmatrix}=0  \right\},                                                
\end{align*}
where we have used \prettyref{lem:transformation_bd_value} in the last equality. This establishes the implication (i) $\Rightarrow$ (ii) with $W_B\coloneqq  \begin{pmatrix}
                                           M & 1
                                           \end{pmatrix}C^{-1} \in \R^{N\times 2N}$.

Assume that $A$ is accretive and that $\dom(A)$ is given as in (ii). Using Theorem \ref{thm:main_congruence}
we find $\mathcal{V}:L^{2}(]a,b[)^{N}\to L^{2}(]-1/2,1/2[)^{N}$ invertible,
such that 
\[
\mathcal{V}P_{1}\partial_{x}\mathcal{V}^{\ast}=\partial_{x}.
\]
We set $D\coloneqq\mathcal{V}A\mathcal{V}^{\ast}$ and obtain an accretive
operator on $L^{2}(]-1/2,1/2[)^{N}$ satisfying $\interior{\partial}_{x}\subseteq D\subseteq\partial_{x}.$
Moreover, by Lemma \ref{lem:transformation_bd_value} the domain of
$D$ is given by 
\[
\dom(D)=\left\{ u\in H^{1}(]-1/2,1/2[)^{N}\,;\,W_{B}C\left(\begin{array}{c}
u(1/2)\\
u(-1/2)
\end{array}\right)=0\right\} 
\]
with $C\in\R^{2N\times2N}$ as in Lemma \ref{lem:transformation_bd_value}.
Moreover, by Theorem \ref{thm:char_accretivity} (a) there exists
$M\in\R^{N}$ with $M^{\ast}M\leq1$ such that 
\[
\dom(D)\subseteq\left\{ u\in H^{1}(]-1/2,1/2[)^{N}\,;\,\left(\begin{array}{cc}
M & 1\end{array}\right)\left(\begin{array}{c}
u(1/2)\\
u(-1/2)
\end{array}\right)=0\right\} .
\]
Consider now the linear mappings 
\begin{align*}
F_{1}:\R^{N}\times\R^{N} & \to H^{1}(]-1/2,1/2[)^{N}\\
(x,y) & \mapsto\left(t\mapsto(1/2-t)y+(1/2+t)x\right)
\end{align*}
and 
\begin{align*}
F_{2}:H^{1}(]-1/2,1/2[)^{N} & \to\R^{N}\times\R^{N}\\
u & \mapsto(u(1/2),u(-1/2)).
\end{align*}
By definition $F_{1}[\ker W_{B}C]\subseteq\dom(D)$ and $F_{2}[\dom(D)]\subseteq\ker\left(\begin{array}{cc}
M & 1\end{array}\right)$ and thus, 
\[
\mathrm{id}=F_{2}\circ F_{1}:\ker W_{B}C\to\ker\left(\begin{array}{cc}
M & 1\end{array}\right)
\]
is a well-defined linear mapping. Hence, 
\[
N\leq\dim\ker W_{B}C\leq\dim\ker\left(\begin{array}{cc}
M & 1\end{array}\right)=N
\]
and so, $\ker W_{B}C=\ker\left(\begin{array}{cc}
M & 1\end{array}\right)$. Thus, there is $L\in\R^{N\times N}$ invertible such that 
\begin{equation}
W_{B}C=L\left(\begin{array}{cc}
M & 1\end{array}\right).\label{eq:WBCLM1}
\end{equation}
 In particular, $W_{B}$ has rank $N$. It is left to show that 
\begin{equation}
W_{B}\left(\begin{array}{cc}
-P_{1} & P_{1}\\
I_{N\times N} & I_{N\times N}
\end{array}\right)^{-1}\left(\begin{array}{cc}
0 & I_{N\times N}\\
I_{N\times N} & 0
\end{array}\right)\left(W_{B}\left(\begin{array}{cc}
-P_{1} & P_{1}\\
I_{N\times N} & I_{N\times N}
\end{array}\right)^{-1}\right)^{*}\geq0\label{eq:posdef}
\end{equation}
and an easy computation shows that \eqref{eq:posdef} is equivalent
to 
\[
W_{B}\left(\begin{array}{cc}
-P_{1}^{-1} & 0\\
0 & P_{1}^{-1}
\end{array}\right)W_{B}^{\ast}\geq0.
\]

Using Lemma \ref{lem:transformation_bd_value}, we infer that 
\[
\left(\begin{array}{cc}
-P_{1}^{-1} & 0\\
0 & P_{1}^{-1}
\end{array}\right)=C^\ast\left(\begin{array}{cc}
-1 & 0\\
0 & 1
\end{array}\right)C
\]
and so
\begin{align*}
W_{B}\left(\begin{array}{cc}
-P_{1}^{-1} & 0\\
0 & P_{1}^{-1}
\end{array}\right)W_{B}^{\ast} & =W_{B}C\left(\begin{array}{cc}
-1 & 0\\
0 & 1
\end{array}\right)C^{\ast}W_{B}^{\ast}\\
 & =L\left(\begin{array}{cc}
M & 1\end{array}\right)\left(\begin{array}{cc}
-1 & 0\\
0 & 1
\end{array}\right)\left(\begin{array}{c}
M^{\ast}\\
1
\end{array}\right)L^{\ast}\\
 & =L(-MM^{\ast}+1)L^{\ast}\geq0
\end{align*}
which shows (ii) $\Rightarrow$ (iii) as well as the formula $W_B C=L\begin{pmatrix} 
                                                                     M & 1
                                                                    \end{pmatrix}$. \\
To complete the proof, we have to show (iii) $\Rightarrow$ (i). We
again set $D\coloneqq\mathcal{V}A\mathcal{V}^{\ast}$ and prove that
$D$ is maximal accretive. We recall that
\[
\dom(D)=\left\{ u\in H^{1}(]-1/2,1/2[)^{N}\,;\,W_{B}C\left(\begin{array}{c}
u(1/2)\\
u(-1/2)
\end{array}\right)=0\right\} 
\]
with $C\in\R^{2N\times2N}$ from Lemma \ref{lem:transformation_bd_value}
and from the computation above, we see that 
\[
W_{B}C\left(\begin{array}{cc}
-1 & 0\\
0 & 1
\end{array}\right)C^{\ast}W_{B}^{\ast}\geq0.
\]
We set $K\coloneqq W_{B}C=\left(\begin{array}{cc}
K_{1} & K_{2}\end{array}\right)\in\R^{N\times2N}$, which has rank $N$ and satisfies $K_{1}K_{1}^{\ast}\leq K_{2}K_{2}^{\ast}$.
Since $K$ has rank $N$, the kernel of $\left(\begin{array}{c}
K_{1}\\
K_{2}
\end{array}\right)$ is trivial. Let now $x\in\ker K_{2}.$ From $K_{1}K_{1}^{\ast}\leq K_{2}K_{2}^{\ast}$
it follows that $x\in\ker K_{1}$ and hence, $x\in\ker\left(\begin{array}{c}
K_{1}\\
K_{2}
\end{array}\right)=\{0\}.$ Thus, $K_{2}$ is invertible and we set $M\coloneqq K_{2}^{-1}K_{1}.$
Then $MM^{\ast}=K_{2}^{-1}K_{1}K_{1}^{\ast}\left(K_{2}^{-1}\right)^{\ast}\leq1$
and hence, also $M^{\ast}M\leq1.$ Moreover, $\ker K=\ker\left(\begin{array}{cc}
M & 1\end{array}\right)$ and thus 
\[
\dom(D)=\left\{ u\in H^{1}(]-1/2,1/2[)^{N}\,;\,\left(\begin{array}{cc}
M & 1\end{array}\right)\left(\begin{array}{c}
u(1/2)\\
u(-1/2)
\end{array}\right)=0\right\} 
\]
and thus, the maximal accretivity of $D$, and hence of $A$, follows from Theorem \ref{thm:char_accretivity}
(c). 
\end{proof}

The formulation of the generator property as in Theorem \ref{thm:well_posedness}
is instrumental to understand the well-posedness theorem for \emph{boundary
control systems} in connection with port-Hamiltonian systems, for
which we will provide a different perspective below.
We recall that in the context of evolutionary equations, the well-posedness
of port-Hamiltonian boundary control systems has already been dealt
with in \cite[Section 5.1]{PTW16_IMA}. In any case, we need the following
notions. Let $W_{B,j}\in\mathbb{R}^{N_{j}\times2N}$ and $N_{j}\in\mathbb{N}$,
$j\in\{1,2\}$, with $N=N_{1}+N_{2},$ and $W_{C}\in\mathbb{R}^{K\times2N}$.
We define
\begin{align*}
\mathfrak{A}\colon\mathrm{dom}(\mathfrak{A})\subseteq L^{2}(]a,b[)^{N} & \to L^{2}(]a,b[)^{N}\\
x & \mapsto P_{1}\left(\mathcal{H}x\right)'+P_{0}\mathcal{H}x,\\
\mathrm{dom}(\mathfrak{A}) & =\left\{x\in L^{2}(]a,b[)^{N}\,;\,\mathcal{H}x\in H^{1}(]a,b[)^{N},W_{B,2}\left(\begin{array}{c}
\mathcal{H}x(b)\\
\mathcal{H}x(a)
\end{array}\right)=0\right\},
\end{align*}

as well as
\begin{align*}
B\colon\mathrm{dom}(\mathfrak{A})\to & \mathbb{R}^{N_{1}}\\
x\mapsto & W_{B,1}\left(\begin{array}{c}
\mathcal{H}x(b)\\
\mathcal{H}x(a)
\end{array}\right),\text{ and }\\
C\colon\mathrm{dom}(\mathfrak{A})\to & \mathbb{R}^{K}\\
x\mapsto & W_{C}\left(\begin{array}{c}
\mathcal{H}x(b)\\
\mathcal{H}x(a)
\end{array}\right).
\end{align*}

\begin{thm}[{\cite[Theorem 2.4]{Zwart2010} and \cite[Theorem 7.7]{Jacob_Zwart_2018}}]
\label{thm:bd_control}Assume that for all $\zeta\in]a,b[$ there
exists an invertible matrix $S(\zeta)$ and a diagonal matrix $\Delta(\zeta)$
such that
\[
P_{1}\mathcal{H}(\zeta)=S(\zeta)^{-1}\Delta(\zeta)S(\zeta),
\]
where$S$ and $\Delta$ are continuously differentiable.
Moreover, assume that $\rk\left(\begin{array}{c}
W_{B,1}\\
W_{B,2}
\end{array}\right)=N$, $\rk\left(\begin{array}{c}
W_{B,1}\\
W_{B,2}\\
W_{C}
\end{array}\right)=N+\rk( W_{C})$ and that $-A\coloneqq-\mathfrak{A}|_{\ker(B)}$ is a generator of
a $C_{0}$-semigroup on $L^{2}(]a,b[)^{N}$. Then for each $\tau>0$
and all $u\in C^{2}([0,\tau])^{N_{1}},$ $x_{0}\in\dom(\mathfrak{A})$
with $Bx_{0}=u(0)$ there exists a unique classical solution $x\in C^{1}([0,\tau])^{N}$
of 
\begin{align*}
\dot{x}(t) & =-\mathfrak{A}x(t)=-P_{1}\left(\mathcal{H}x\right)'-P_{0}\mathcal{H}x,\quad x(0)=x_0\\
u(t) & =Bx(t)=W_{B,1}\left(\begin{array}{c}
\mathcal{H}x(t,b)\\
\mathcal{H}x(t,a)
\end{array}\right),\\
y(t) & =Cx(t)=W_{C}\left(\begin{array}{c}
\mathcal{H}x(t,b)\\
\mathcal{H}x(t,a)
\end{array}\right).
\end{align*}
Moreover, there exists a constant $m_{\tau}\geq0$ (just depending
on $\tau$) such that 
\[
\|x(\tau)\|_{L^{2}(]a,b[)^{N}}^{2}+\int_{0}^{\tau}\|y(t)\|_{L^{2}(]a,b[)^{K}}^{2}dt\leq m_{\tau}\left(\|x_{0}\|_{L^{2}(]a,b[)^{N}}^{2}+\int_{0}^{\tau}\|u(t)\|_{L^{2}(]a,b[)^{N_{1}}}^{2}\right).
\]
\end{thm}

Owing to the flexibility of evolutionary equations, we are able to
significantly improve the well-posedness result in as much as we do
not need to impose any regularity conditions on  $\mathcal{H}$.
Further, we can address systems which are more general than the Cauchy problems of the previous theorem. In particular we consider
differential-algebraic equations. For this we  consider equations of the following form.

\begin{defn}
\label{def:DAE}Let $P_{1}=P_{1}^{*}\in\mathbb{R}^{N\times N}$ be
invertible and $\mathcal{H}\in L^{\infty}([a,b];\mathbb{R}^{N\times N})$.
Assume there exist $m,M\in\mathbb{R}_{>0}$ such that 
\[
mI_{N\times N}\leq\mathcal{H}(\zeta)=\mathcal{H}(\zeta)^{*}\leq MI_{N\times N}\quad(\text{a.e. }\zeta\in[a,b]).
\]
Let $M_{0}=M_{0}^{*},M_{1}\in L(L^{2}(]a,b[)^{N})$ such that $M_{0}\geq0$.
An equation of the form
\[
\left(\partial_{t}M_{0}+M_{1}+P_{1}\partial_{x}\right)\mathcal{H}X=F,
\]
where $F\colon\mathbb{R}\times]a,b[\to\mathbb{R}^{N}$ is given
and $X\colon\mathbb{R}\times]a,b[\to\mathbb{R}^{N}$ is the unknown,
is called a \emph{differential-algebraic port-Hamiltonian equation};
here $\partial_{t}$ is the derivative with respect to the $\mathbb{R}$-variable
(`time') and $\partial_{x}$ is the coordinate wise derivative with
respect to the spatial variable in $]a,b[$. 
\end{defn}

\begin{rem}
The classical port-Hamiltonian operator is then covered by choosing
$M_{1}=P_{0}$ and $M_{0}=\mathcal{H}^{-1}$.
\end{rem}

We provide the counterpart of the generation property first. We will
restrict ourselves to maximal accretive restrictions of $P_{1}\partial_{x}$
(corresponding to the case of a generator of a contraction semigroup)
and note that the additional generality enters the problem via the
operators $M_{0}$ and $M_{1}$. In order to formulate the well-posedness
result, we need some notations from the theory of evolutionary equations.
\begin{defn}
For a Hilbert space $H$ and $\rho>0$ we define 
\[
L_{\rho}^{2}(\R;H)\coloneqq\left\{ u:\R\to H\,;\,u\text{ measurable},\,\int_{\R}\|u(t)\|^{2}\e^{-2\rho t}\d t<\infty\right\} 
\]
equipped with the natural inner product. Moreover, we define the operator
$\partial_{t}$ on $L_{\rho}^{2}(\R;H)$ as the closure of 
\begin{align*}
C_{c}^{1}(\R;H)\subseteq L_{\rho}^{2}(\R;H) & \to L_{\rho}^{2}(\R;H)\\
\phi & \mapsto\phi',
\end{align*}
where $C_{c}^{1}(\R;H)$ denotes the space of continuously differentiable
functions with compact support on $\R$ taking values in $H$. 
\end{defn}

We remark that the so-defined operator $\partial_{t}$ is continuously
invertible on $L_{\rho}^{2}(\R;H)$ (see e.g.~\cite{Picard_McGhee2011,ISEM2020})
and thus, allows for the following definition.
\begin{defn}
For $k\in\N$ and $\rho>0$ we define the space $H_{\rho}^{-k}(\R;H)$
as the completion of $L_{\rho}^{2}(\R;H)$ with respect to the norm
\[
\|u\|_{\rho,-k}\coloneqq\|\partial_{t}^{-k}u\|_{L_{\rho}^{2}}.
\]
\end{defn}

\begin{rem}
It is easy to see that $\partial_{t}$ can be extended to a continuously
invertible operator on $H_{\rho}^{-k}(\R;H)$ and that each closed
densely defined operator between two Hilbert spaces $H_{0},H_{1}$
can be canonically extended to a closed and densely defined operator
between $H_{\rho}^{-k}(\R;H_{0})$ and $H_{\rho}^{-k}(\R;H_{1}).$ 
\end{rem}

We recall the well-posedness theorem for evolutionary equations in
the form needed here; see also \cite[Chapter 6]{ISEM2020}.
\begin{thm}[{\cite[Solution Theory]{Picard2009}, \cite[Theorem 6.2.5]{Picard_McGhee2011}}]
\label{thm:sol_theory_evo} Let $H$ be a Hilbert space, $M_{0},M_{1}\in L(H)$
with $M_{0}=M_{0}^{\ast}\geq0$ and assume there exists $\rho_{0}>0$
and $c>0$ such that 
\[
\rho_{0}\langle M_{0}x,x\rangle+\Re\langle M_{1}x,x\rangle\geq c\|x\|^{2}
\]
for all $x\in H$. Moreover, let $A:\dom(A)\subseteq H\to H$ be maximal
accretive. Then for each $\rho\geq\rho_{0}$ the operator 
\[
\left(\overline{\partial_{t}M_{0}+M_{1}+A}\right)
\]
is continuously invertible on $H_{\rho}^{-k}(\R;H)$ for each $k\in\mathbb{N}.$ 
\end{thm}

Before we can state the well-posedness result for boundary control problems for differential-algebraic port-Hamiltonian equations, we need the following prerequisit. 

\begin{lem}
\label{lem:surjectivity_trace}Let $N\in\mathbb{N}$. Then there exists
$\eta\colon\mathbb{R}^{2N}\to H^{1}(]a,b[)^{N}$ continuous such that
\[
\gamma\left(\eta(v)\right)=v
\]
for all $v\in\mathbb{R}^{2N}$, where 
\[
\gamma\colon H^{1}(]a,b[)^{N}\ni(x_{k})_{k\in\{1,\ldots,N\}}\mapsto\left(\begin{array}{c}
\left(x_{k}(b-)\right)_{k\in\{1,\ldots,N\}}\\
\left(x_{k}(a+)\right)_{k\in\{1,\ldots,N\}}
\end{array}\right).
\]
\end{lem}

\begin{proof}
Let $v=(v_{1},v_{2})\in\R^{N}\times\R^{N}$. Then
\[
\eta(v)(t)\coloneqq\frac{1}{b-a}\left((t-a)v_{1}+(b-t)v_{2}\right)
\]
is a valid choice for $\eta$. The continuity properties are easily
checked.
\end{proof}

The next result puts Theorem \ref{thm:bd_control} into perspective
of evolutionary equations. Note that we do not assume any regularity
condition on $\mathcal{H}$. As our main assumption, we shall assume
the accretivity of the `derivative part' of the port-Hamiltonian. 
\begin{thm}
Consider a differential-algebraic port-Hamiltonian equation as in Definition
\ref{def:DAE}. Assume that there
exists $\rho_{0}\geq0$ such that for all $x\in L^{2}(]a,b[)^{N}$
\[
\rho_{0}\langle M_{0}x,x\rangle+\Re\langle M_{1}x,x\rangle\geq c\langle x,x\rangle
\]
for some $c>0$. Let $\gamma\colon H^{1}(]a,b[)^{N}\to\mathbb{R}^{2N}$
be given by 
\[
\gamma(x_{k})_{k\in\{1,\ldots,N\}}\mapsto\left(\begin{array}{c}
\left(x_{k}(b-)\right)_{k\in\{1,\ldots,N\}}\\
\left(x_{k}(a+)\right)_{k\in\{1,\ldots,N\}}
\end{array}\right).
\]
Let $W\in\mathbb{R}^{N\times2N}$ be such that 
\begin{align*}
A\colon\mathrm{dom}(A)\subseteq L^{2}(]a,b[)^{N} & \to L^{2}(]a,b[)^{N}\\
x & \mapsto P_{1}\partial_{x}x
\end{align*}
with 
\[
\mathrm{dom}(A)=\{x\in H^{1}(]a,b[)^{N}\,;\,W\gamma x=0\}
\]
is maximal accretive (cp.~Theorem \ref{thm:well_posedness}). Furthermore,
let $u\in H_{\rho}^{-k}(\mathbb{R};\mathbb{R}^{N})$ for some $k\in\mathbb{N}$
and $\rho\ge\rho_{0}$. Then there exists a unique $x\in H_{\rho}^{-k-1}(\mathbb{R};L_{\mathcal{H}}^{2}(]a,b[)^{N})$
such that 
\begin{align*}
\left(\partial_{t}M_{0}+M_{1}+P_{1}\partial_{x}\right)\mathcal{H}x & =0\\
W\gamma\mathcal{H}x & =u.
\end{align*}
Moreover, the mapping 
\[
H_{\rho}^{-k}(\mathbb{R};\mathbb{R}^{N})\ni u\mapsto\mathcal{H}x\in H_{\rho}^{-k-1}(\mathbb{R};L^{2}(]a,b[)^{N})\cap H_{\rho}^{-k-2}(\mathbb{R};H^{1}(]a,b[)^{N})
\]
is continuous. In particular, 
\[
u\mapsto y\coloneqq C\gamma\mathcal{H}x\in H_{\rho}^{-k-2}(\mathbb{R};\mathbb{R}^{K})
\]
is continuous for each $C\in\R^{K\times2N}$.
\end{thm}

\begin{rem}
The continuity statements in the previous theorem are the proper replacements
for the inequality asserted to hold in Theorem \ref{thm:bd_control}.
We emphasise that the previous theorem also deals with differential-algebraic
equations as well as with rough $\mathcal{H}$. The price we have
to pay is the regularity loss of the solution $x$ and the observation
$y$. 
\end{rem}

\begin{proof}
We consider the equation 
\begin{align*}
\left(\partial_{t}M_{0}+M_{1}+P_{1}\partial_{x}\right)\mathcal{H}x & =0\\
W\gamma\mathcal{H}x & =u.
\end{align*}

Using that $A$ is m-accretive, we apply Theorem \ref{thm:well_posedness}
to find an invertible matrix $C\in\R^{2N\times2N}$ and two matrices
$L,M\in\R^{N\times N}$ such that $L$ is invertible and $M^{\ast}M\leq1$
with 
\[
WC=L\left(\begin{array}{cc}
M & 1\end{array}\right).
\]
Then 
\[
W\gamma\mathcal{H}x=u
\]
is equivalent to 
\begin{align*}
0 & =L\left(\begin{array}{cc}
M & 1\end{array}\right)C^{-1}\gamma\mathcal{H}x-u\\
 & =L\left(\begin{array}{cc}
M & 1\end{array}\right)\left(C^{-1}\gamma\mathcal{H}x-\left(\begin{array}{c}
0\\
L^{-1}u
\end{array}\right)\right)\\
 & =W\left(\gamma\mathcal{H}x-C\left(\begin{array}{c}
0\\
L^{-1}u
\end{array}\right)\right)
\end{align*}

Let $\eta$ be as in Lemma \ref{lem:surjectivity_trace}. 
Then the latter can equivalently be formulated by 
\begin{align*}
0 & =W\left(\gamma\mathcal{H}x-C\left(\begin{array}{c}
0\\
L^{-1}u
\end{array}\right)\right)\\
 & =W\left(\gamma\mathcal{H}x-\gamma\eta C\left(\begin{array}{c}
0\\
L^{-1}u
\end{array}\right)\right)\\
 & =W\gamma\left(\mathcal{H}x-\eta C\left(\begin{array}{c}
0\\
L^{-1}u
\end{array}\right)\right)
\end{align*}
which in turn is equivalent to 
\[
\mathcal{H}x-\tilde{u}\in\dom(A),
\]
where $\tilde{u}=\eta C\left(\begin{array}{c}
0\\
L^{-1}u
\end{array}\right)\in H_{\rho}^{-k}(\R;H^{1}(]a,b[)^{N})$. Thus, we obtain
\begin{align*}
\left(\partial_{t}M_{0}+M_{1}+P_{1}\partial_{x}\right)\mathcal{H}x & =0\\
W\gamma\mathcal{H}x & =u
\end{align*}
amounts to asking for 
\begin{align}
\overline{\left(\partial_{t}M_{0}+M_{1}+A\right)}(\mathcal{H}x-\tilde{u}) & =-\left(\partial_{t}M_{0}+M_{1}+P_{1}\partial_{x}\right)\tilde{u}\in H_{\rho}^{-k-1}(\R;L^{2}(]a,b[)^{N}).\label{eq:problem_aux}
\end{align}
Next, Theorem \ref{thm:sol_theory_evo} leads to unique existence
of $\mathcal{H}x-\tilde{u}\in H_{\rho}^{-k-1}(\R;L^{2}(]a,b[)^{N})$,
which shows the unique existence of $x\in H_{\rho}^{-k-1}(\R;L_{\mathcal{H}}^{2}(]a,b[)^{N})$
(the computation above shows existence and performing the steps backwards,
we obtain uniqueness) solving the problem. Moreover, since the mapping
\[
H_{\rho}^{-k}(\R;\R^{N})\ni u\mapsto\tilde{u}\in H_{\rho}^{-k}(\R;H^{1}(]a,b[)^{N})
\]
is continuous by Lemma \ref{lem:surjectivity_trace} and 
\[
H_{\rho}^{-k}(\R;H^{1}(]a,b[)^{N})\ni\tilde{u}\mapsto\left(\partial_{t}M_{0}+M_{1}+P_{1}\partial_{x}\right)\tilde{u}\in H_{\rho}^{-k-1}(\R;L^{2}(]a,b[)^{N})
\]
is easily seen to be continuous, Theorem \ref{thm:sol_theory_evo}
yields the continuity of 
\[
H_{\rho}^{-k}(\R;\R^{N})\ni u\mapsto\mathcal{H}x\in H_{\rho}^{-k-1}(\R;L^{2}(]a,b[)^{N}).
\]

Moreover, by (\ref{eq:problem_aux}) we have that 
\[
A(\mathcal{H}x-\tilde{u})=-\left(\partial_{t}M_{0}+M_{1}+P_{1}\partial_{x}\right)\tilde{u}-(\partial_{t}M_{0}+M_{1})(\mathcal{H}x-\tilde{u})\in H_{\rho}^{-k-2}(\R;L^{2}(]a,b[)^{N})
\]
which yields the continuity of the mapping 
\[
H_{\rho}^{-k}(\R;\R^{N})\ni u\mapsto\mathcal{H}x\in H_{\rho}^{-k-2}(\R;H^{1}(]a,b[)^{N})
\]
since $\tilde{u}\in H_{\rho}^{-k}(\R;H^{1}(]a,b[)^{N})\subseteq H_{\rho}^{-k-2}(\R;H^{1}(]a,b[)^{N})$. 
\end{proof}

\end{document}